\documentclass[envcountsect,envcountsame,reqno,a4paper]{llncs}

\usepackage{
amsmath,
amsfonts,
amssymb,
amscd,
calc, 
enumerate, 
}

\newenvironment{fxthm}[1]{
\begin{trivlist}
\item {\bf #1}\;
}
{\end{trivlist}}

\newcommand{\darrow}{\!\downarrow}
\newcommand{\uarrow}{\!\uparrow}

\renewcommand{\leq}{\leqslant}
\renewcommand{\geq}{\geqslant}

\newcommand{\fa}{\forall}
\newcommand{\ex}{\exists}
\newcommand{\vph}{\varphi}

\newcommand{\A}{\mathcal{A}}
\newcommand{\K}{\mathcal{K}}
\newcommand{\FPF}{\mathrm{FPF}}

\begin{document}

\mainmatter  

\title{Fixed Point Theorems in Computability Theory}

\titlerunning{Fixed point theorems}

\author{Sebastiaan A. Terwijn
}
\authorrunning{Sebastiaan A. Terwijn} 

\institute{Radboud University Nijmegen\\
Department of Mathematics\\
P.O. Box 9010, 6500 GL Nijmegen\\ 
the Netherlands\\
\email{terwijn@math.ru.nl}
}

\tocauthor{Sebastiaan Terwijn (Radboud University Nijmegen)}

\maketitle
\thispagestyle{plain} 

\begin{abstract}
We give a quick survey of the various fixed point theorems 
in computability theory, partial combinatory algebra, and 
the theory of numberings, as well as generalizations based 
on those. We also point out several open problems connected 
to these.
\end{abstract}




\section{Introduction}

Let $\vph_e$ denote the $e$-th partial computable (p.c.) function. 
Then Kleene's recursion theorem 
simply states that for 
every computable function $f$ there exists a number $e\in\omega$ 
such that 
$$
\vph_{f(e)} = \vph_{e}.
$$
We can think of $e$ as a fixed point of $f$, not literally in the
sense that $f(e)=e$, but at the level of codes of p.c.\ functions. 

The proof of the recursion theorem is very short, and 
it has an air of mystery. 
In Kleene's original paper \cite{Kleene}, which is about ordinal 
notations, it is somewhat hidden at the end of section~2, 
where it only occupies two cryptic lines, but even when 
written out it can be done in three or four lines
(taking the $S$-$m$-$n$-theorem for granted). 
The fact that the proof is not very illuminating is 
perhaps due to the fact that it does not occur 
naturally in the context of p.c.\ functions. Kleene 
found the fixed point theorem in the lambda calculus, 
where it does occur in a natural way, and then translated 
it to computability theory to obtain the recursion theorem.\footnote{
Explanations of the recursion theorem, elaborating on its short 
proof, are given in Owings~\cite{Owings} and Odifreddi~\cite{Odifreddi}.}
The following analogy between lambda calculus and computability
theory may be helpful. To simplify matters we write 
$n\sim m$ for $\vph_n = \vph_m$, and we define (partial) application 
of numbers as $nm = \vph_n(m)$. 
The left part of the following table follows  
Barendregt~\cite[2.1.5]{Barendregt}.
\begin{center}
\begin{tabular*}{\textwidth}[t]{@{\hspace*{1.8cm}}p{5.5cm}l}
{\bf $\lambda$-calculus}  &  {\bf computability} \\[6pt]
$\lambda$-terms & $n\in\omega$ \\[6pt]
$FG$            & $n m = \vph_n(m)$ \\[6pt]
\underline{fixed point theorem:}  &  \underline{recursion theorem:} \\[6pt]
$\fa F \, \ex X \; FX=X$  & $\fa f \, \ex x \; fx \sim x$  \\[6pt]
{\em Proof:} $W = \lambda x. F(xx)$  & {\em Proof:} $bx \sim f(xx)$ \\[6pt]
$WW = F(WW)$ \hspace*{.3cm} $\square$ &  
$bb \sim f(bb)$ \hspace*{.9cm} $\square$ 
\end{tabular*} 
\end{center}
That there exists $b$ such that $bx \sim f(xx)$ 
(i.e.\ $\vph_{\vph_b(x)} = \vph_{f(\vph_x(x))}$)
follows from the $S$-$m$-$n$-theorem.

The recursion theorem is a fundamental result of computability theory
that has found many applications, even extending beyond pure 
computability theory, e.g.\ in set theory. 
The first application was to develop a theory of constructive 
ordinals, for which the recursion theorem is indispensable.
The theorem and its many applications were excellently reviewed by 
Kleene's student Yiannis Moschovakis in \cite{Moschovakis2010}.
The present quick survey is by no means intended to replace
that much larger one, but we have a slightly different focus, 
especially with the view from combinatory algebra, and we 
will also discuss more recent results. 
We will not be including any proofs (apart from the proof 
above of the recursion theorem itself, and some short arguments), 
but mainly pointers to the literature. 
We will also not discuss applications, 
which is the main focus of~\cite{Moschovakis2010}.

Our notation for basic notions in computability theory is mostly 
standard. 
As already mentioned, $\vph_e$ denotes the $e$-th partial computable
(p.c.) function, in some standard numbering of the p.c.\ functions.
P.c.\ functions are denoted by lower case
Greek letters, and (total) computable functions by lower case
Roman letters.
$\omega$ denotes the natural numbers.
$W_e$ denotes the $e$-th computably enumerable (c.e.) set, which is 
defined as the domain of $\vph_e$. 
We write $\vph_e(n)\darrow$ if this computation is defined,
and $\vph_e(n)\uarrow$ otherwise.
$\emptyset'$ denotes the halting set.
For unexplained notions we refer to Odifreddi~\cite{Odifreddi} or
Soare~\cite{Soare}.
Our presentation of partial combinatory algebra follows 
van Oosten~\cite{vanOosten}.

\section{The second recursion theorem}\label{sec:srt}

The simple version of the recursion theorem stated at the 
beginning of this paper is fully effective, which means 
that we can compute the fixed point $e$ effectively from 
a code of~$f$. One way to state this is as follows:
Let $h(x,n)$ be a computable binary function.
By the recursion theorem, for every choice of $n$ there 
exists $e$ such that $\vph_{e} = \vph_{h(e,n)}$, and by 
Skolemization we can see $e$ as a function $f(n)$ of~$n$. 
Now the effectiveness means that we can choose $f$ to be 
{\em computable\/}, so that 
$$
\vph_{f(n)} = \vph_{h(f(n),n)}
$$
for every~$n$.
This is called the recursion theorem with parameters, or 
the {\em second\/} recursion theorem.\footnote{
The term second recursion theorem is from Rogers, and 
advocated by Moschovakis. 
Confusingly, Kleene called this form the first recursion theorem, 
a term which is now mostly used to refer to the simple version 
without parameters (also following Rogers).}
Another way to phrase this is as follows:\footnote{
The two forms of the recursion theorem with parameters are 
equivalent because the numbering $n\mapsto \vph_n$ of the 
p.c.\ functions is {\em precomplete}, cf.\ footnote~\ref{ft}.}

\begin{theorem}{\rm (The second recursion theorem, Kleene \cite{Kleene})}
\label{recthm2}
There exists a computable function~$f$
such that for every~$n$, if $\vph_n(f(n))\darrow$ then
$$
\vph_{\vph_n(f(n))} = \vph_{f(n)}.
$$
\end{theorem}

\section{Partial combinatory algebra}

Partial combinatory algebra was first introduced in the literature
by Feferman~\cite{Feferman} as an abstract axiomatic model of computation, 
though the concept had been known and discussed before. 
A {\em partial combinatory algebra\/} (pca) is a set $\A$ with 
a partial application operator $\cdot$ from $\A\times \A$ to $\A$. 
Instead of $a\cdot b$ we often simply write $ab$. We write 
$ab\darrow$ if this is defined. 
By convention application associates to the left, so $abc$ should 
be read as $(ab)c$.
We call $f\in\A$ {\em total\/} if $fa\darrow$ for every~$a$.
For terms (i.e.\ expressions built from elements of $\A$, variables, 
and application) $t$ and $s$ we write $t\simeq s$ if either both sides are 
undefined, or defined and equal.
The defining property of a pca
is that it should be {\em combinatory complete\/}, that is, 
for any term $t(x_1,\ldots,x_n,x)$, $n\geq 0$, 
there exists a $b{\in} \A$ such that
for all $a_1,\ldots,a_n,a{\in} \A$,
\begin{enumerate}[\rm (i)]

\item $ba_1\cdots a_n\darrow$,

\item $ba_1\cdots a_n a \simeq t(a_1,\ldots,a_n,a)$.

\end{enumerate}
Combinatory completeness is equivalent to the existence of the 
combinators $s$ and $k$, familiar from combinatory algebra, 
cf.\ van Oosten~\cite{vanOosten}.

Feferman proved the following version of the recursion theorem
in pcas:

\begin{theorem} {\rm (Feferman \cite{Feferman})}
\label{Feferman}
Let $\A$ be a pca. 
\begin{enumerate}[\rm (1)]

\item There exists $f{\in} \A$ such that for all $g{\in} \A$, 
$g(fg) \simeq fg$.

\item There exists a total $f{\in} \A$ such that  
$g(fg)a \simeq fga$ for every $g$ and $a{\in} \A$.

\end{enumerate}
\end{theorem}

The prime example of a pca is $\omega$, with application 
$nm = \vph_n(m)$ as already defined above. 
From this we immediately recognize Theorem~\ref{Feferman}~(2) 
as a generalization of Theorem~\ref{recthm2}.
However, there is a rich variety of other examples of pcas, 
drawing from lambda calculus, constructive mathematics, 
realizability, and computability theory. 
Examples and references may be found for example in 
van Oosten~\cite{vanOosten}, 
Cockett and Hofstra \cite{CockettHofstra},    
Longley and Normann~\cite{LongleyNormann}, and 
Golov and Terwijn~\cite{GolovTerwijn2023}.

\section{The theory of numberings}

A {\em numbering\/} of a set $S$ is a surjection
$\gamma\colon\omega\rightarrow S$. 
For every numbering we have an equivalence relation on $\omega$ 
defined by $n \sim_\gamma m$ if $\gamma(n) = \gamma(m)$. 
A numbering $\gamma$ is {\em precomplete\/} if for every p.c.\ 
function $\psi$ there exists a computable function $f$
such that for every~$n$
\begin{equation} \label{precomplete}
\psi(n)\darrow \; \Longrightarrow \; f(n) \sim_\gamma \psi(n). 
\end{equation}
Following Visser, we say that 
{\em $f$ totalizes $\psi$ modulo~$\sim_\gamma$\/}.

Ershov~\cite{Ershov2} proved that the recursion theorem holds for every 
precomplete numbering: If $\gamma$ is precomplete, then for every 
computable function $f$ there exists $e\in\omega$ 
such that 
\begin{equation}\label{Ershov1}
f(e) \sim_\gamma e.
\end{equation}
The recursion theorem is obtained from this by simply taking 
the numbering $n\mapsto\vph_n$ of the p.c.\ functions. 
This numbering is easily seen to be precomplete
by the $S$-$m$-$n$-theorem.

%

As for the recursion theorem, we have a version of 
Ershov's recursion theorem with parameters, which shows 
that the theorem is effective. 
The following formulation from Andrews, Badaev, and Sorbi~\cite{ABS}
is completely analogous to Theorem~\ref{recthm2} above.\footnote{
An alternative way to state Ershov's recursion theorem is:
For every computable function $h(x,n)$ there is a computable 
function $f$ such that for all $n$, 
$$
f(n)\sim_\gamma h(f(n),n).
$$
The two forms are equivalent {\em for precomplete numberings}, 
see the discussion in 
Barendregt and Terwijn \cite[Section 3]{BarendregtTerwijn}.
Question~3.4 there asks if for arbitrary numberings $\gamma$ 
the equivalence implies that $\gamma$ is precomplete. 
This question is still open. A partial answer was obtained in 
\cite{GolovTerwijn2022}, where it was shown that the answer for 
the relativized version of this question is negative.\label{ft}}

\begin{theorem} \label{Ershov}
{\rm (Ershov's recursion theorem \cite{Ershov2})}
Let $\gamma$ be a precomplete numbering.
There exists a computable function~$f$
such that for every~$n$, if $\vph_n(f(n))\darrow$ then
$$
\vph_n(f(n)) \sim_\gamma f(n).
$$
\end{theorem}

In order to combine the theorems of Feferman and Ershov, we 
consider {\em generalized numberings\/} 
$\gamma\colon \A \rightarrow S$, having as a base a pca $\A$ 
instead of $\omega$. We call such numberings 
{\em precomplete\/}\footnote{
In \cite{BarendregtTerwijn} precompleteness was defined using terms 
instead of elements $b\in\A$, but the two definitions are 
equivalent by \cite[Lemma 6.4]{BarendregtTerwijn}.}
if every $b\in\A$ can be totalized modulo~$\sim_\gamma$, 
similarly to the definition of precompleteness 
\eqref{precomplete} for ordinary numberings, namely if 
for every $b\in\A$
there exists a total element $f\in \A$ such that
for all $a\in \A$,
\begin{equation*} 
ba\darrow \; \Longrightarrow \; fa \sim_\gamma ba. 
\end{equation*}

\begin{theorem} {\rm (Barendregt and Terwijn \cite{BarendregtTerwijn})} 
\label{recthmpca}
Suppose $\A$ is a pca, and that
$\gamma\colon \A \rightarrow S$ is a precomplete generalized numbering.
Then there exists a total $f\in \A$ such that for all $g\in \A$,
if $g(fg)\darrow$ then
$$
g(fg) \sim_\gamma fg.
$$
\end{theorem}

\section{Overview}

In this section we list the various forms of the recursion theorem
discussed so far.
To ease the comparison, we write them as succinctly as possible.

First consider the natural numbers $\omega$ as a pca, 
with application $nm = \vph_n(m)$.

\begin{fxthm}{Kleene 1}
$\fa n \, \ex m \, \fa a \; (nma \simeq ma)$.
\end{fxthm}

\begin{fxthm}{Kleene 2}
$\ex f \text{ total } \fa n \, \fa a \; (n(fn)a \simeq fna)$.
\end{fxthm}

Let $\gamma:\omega \rightarrow S$ be a precomplete numbering.

\begin{fxthm}{Ershov}
$\ex f \text{ total } \fa n \; (n(fn)\darrow \Longrightarrow n(fn) \sim_\gamma fn)$.
\end{fxthm}

Let $\A$ be a pca.

\begin{fxthm}{Feferman 1}
$\ex f \fa g \; (g(fg) \simeq fg)$.
\end{fxthm}

\begin{fxthm}{Feferman 2}
$\ex f \text{ total } \fa g \, \fa a \; (g(fg)a \simeq fga)$.
\end{fxthm}

Let $\gamma:\A \rightarrow S$ be a precomplete generalized numbering.

\begin{fxthm}{BT}
$\ex f \text{ total } \fa g \; (g(fg)\darrow \Longrightarrow g(fg) \sim_\gamma fg)$.
\end{fxthm}

We have the following relations between these.

\medskip\noindent
{\bf Kleene 2} $\Rightarrow$ {\bf Kleene 1}: This is obvious, since 
$fn$ provides the fixed point.

\medskip\noindent
{\bf Ershov} $\Rightarrow$ {\bf Kleene 2}: The numbering $\gamma: n\mapsto\vph_n$ is
precomplete by the $S$-$m$-$n$-theorem, and
$n(fn) \sim_\gamma fn$ iff $\fa a \; (n(fn)a \simeq fna)$.

\medskip\noindent
{\bf Feferman 2} $\Rightarrow$ {\bf Kleene 2}: Immediate from the fact that 
$\omega$ is a pca. 

\medskip\noindent
{\bf BT} $\Rightarrow$ {\bf Ershov}: This is trivial.

\medskip\noindent
{\bf BT} $\Rightarrow$ {\bf Feferman 2}: 
Let $a \sim_e b$ if $\fa x{\in}\A (ax \simeq bx)$.
Then the natural map $\gamma: \A \rightarrow \A/{\sim_e}$ is a 
precomplete generalized numbering by 
\cite[Proposition 4.2]{BarendregtTerwijn2022}.
Applying BT to this numbering gives Feferman 2.

\medskip\noindent
{\bf Feferman 1} by itself is very weak, and does not directly imply anything.

\medskip
The implications above are summarized in Figure~\ref{fig}.

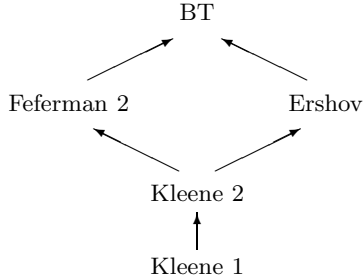
\begin{figure}[h]
\begin{center}
\setlength{\unitlength}{.8mm}
\begin{picture}(24, 45)
\put(12,0){\makebox[0cm][c]{Kleene 1}}
\put(12,4){\vector(0,2){6}}
\put(12,12){\makebox[0cm][c]{Kleene 2}}
\put(9,17){\vector(-2,1){14}}
\put(15,17){\vector(2,1){14}}
\put(-6,27){\makebox[-.5cm][c]{Feferman 2}}
\put(-6,32){\vector(2,1){14}}
\put(30,27){\makebox[.5cm][c]{Ershov}}
\put(30,32){\vector(-2,1){14}}
\put(12,42){\makebox[0cm][c]{BT}}
\end{picture}
\end{center}
\caption{The relation between various forms of the recursion theorem\label{fig}}
\end{figure}

\section{Fixed point free functions and Arslanov's completeness criterion}

Recall that $W_n$ is the $n$-th c.e.\ set.
A function $f$ is called {\em fixed point free}, or simply FPF,
if $W_{f(n)}\neq W_n$ for every~$n$.
Note that by the recursion theorem no FPF function is computable.
We will also consider {\em partial\/} functions without fixed points.
Extending the above definition, we call a partial function $\delta$
FPF if for every $n$,
\begin{equation}
\delta(n)\darrow \; \Longrightarrow \; W_{\delta(n)}\neq W_n. 
\end{equation}
Below, by FPF function we will always mean a total function,
unless explicitly stated otherwise.

The standard tool in computability theory to measure the complexity of sets 
is the notion of a Turing reduction. Informally, $B\leq_T A$ means that $A$
can compute $B$. 
Here we are interested in the complexity of computing FPF functions.
The following fact is well-known:

\begin{proposition} {\rm (Jockusch et al.\ \cite{Jockuschetal})} \label{equivalence}
The following are equivalent for any set~$A$:
\begin{enumerate}[\rm (i)]

\item $A$ computes a $\FPF$ function,

\item $A$ computes a function $h$ such that
$\vph_{h(e)} \neq \vph_e$ for every~$e$.

\end{enumerate}

\end{proposition}
\begin{proof}
(i)$\Rightarrow$(ii) is trivial. 
We give a direct proof of (ii)$\Rightarrow$(i), avoiding the
detour via DNC functions as in Soare~\cite[p90]{Soare}.

Let $\psi$ be p.c.\ such that $W_e\neq\emptyset \Rightarrow \psi(e)\in W_e$.
Let $p$ totalize $\psi$, i.e.\ $p$ is computable such that
$\vph_{p(e)} = \vph_{\psi(e)}$ for every $e$ with $\psi(e)\darrow$.
(As mentioned before, every p.c.\ function can be totalized in this way.)
Now suppose that $h$ is as in (ii), and let $f$ be $A$-computable such that
$W_{f(e)} = \{h(p(e))\}$.

Suppose that $W_{f(e)} = W_e$. Then $\psi(e)\darrow = h(p(e))$,
hence
$$
\vph_{h(p(e))} = \vph_{\psi(e)} = \vph_{p(e)}, 
$$
contradicting (ii).
Hence $f$ is a FPF function.\qed
\end{proof}

Arslanov (building on earlier work of Martin and Lachlan) 
extended the recursion theorem from computable functions
to functions computable from an incomplete c.e.\ Turing degree.
The Arslanov completeness criterion states that
a c.e.\ set is Turing complete if and only if it
computes a fixed point free function.

\begin{theorem} {\rm (Arslanov completeness criterion \cite{Arslanov})}
\label{Arslanov}
Suppose $A$ is c.e.\ and $A$ is Turing incomplete, i.e.\ $A <_T \emptyset'$.
If $f$ is an $A$-computable function, then $f$ has a fixed point,
i.e.\ an $e\in\omega$ such that $W_{f(e)} = W_e$.
\end{theorem}

Note that Theorem~\ref{Arslanov} implies the recursion theorem
by Proposition~\ref{equivalence}.
Without the requirement that $A$ is c.e.\ the theorem fails, 
as $\FPF$ functions can have low Turing degree 
(i.e.\ $A' \leq_T \emptyset'$) 
by the low basis theorem of Jockusch and Soare~\cite{JockuschSoare1972b}.

The Arslanov completeness criterion  
has been extended in various ways,
by considering relaxations of the type of fixed point.  
For example, instead of requiring that $W_{f(e)} = W_e$, 
we can merely require $W_{f(e)}$ to be a finite variant of $W_e$ (Arslanov), 
or for them to be Turing equivalent (Arslanov),
or for the $n$-th jumps of these sets to be Turing equivalent (Jockusch).
In this way the completeness criterion can be extended to all levels
of the arithmetical hierarchy.
For a discussion of these results we refer the reader to
Soare \cite[p270 ff]{Soare} and
Jockusch, Lerman, Soare, and Solovay~\cite{Jockuschetal}.
The latter paper  
also contains an extension of Theorem~\ref{Arslanov}
from c.e.\ degrees to d.c.e.\ degrees.

\section{Further generalizations}

In this section we discuss several other generalizations 
of the recursion theorem.

Visser proved an extension called the ADN theorem
(for ``anti diagonal normalization theorem''), 
motivated by Rosser's extension of G\"odel's incompleteness theorem.

\begin{theorem} {\rm (ADN theorem, Visser~\cite{Visser})} \label{ADN}
Suppose that $\delta$ is a partial computable fixed point free
function. Then for every partial computable function $\psi$ there
exists a computable function $f$ such that for every $n$,
\begin{align}
\psi(n)\darrow \; &\Longrightarrow \; W_{f(n)}= W_{\psi(n)} \label{totalize} \\
\psi(n)\uarrow \; &\Longrightarrow \; \delta(f(n))\uarrow   \label{avoid}
\end{align}
\end{theorem}
Note that \eqref{totalize} expresses that $f$ totalizes $\psi$ 
modulo the numbering $n\mapsto W_n$ of the c.e.\ sets.
Also note that the ADN theorem implies the recursion theorem: 
The function $\delta$ cannot be total, for otherwise $f(n)$ could 
not exist when $\psi(n)\uarrow$. 
It follows that there can be no computable FPF function. 
By Proposition~\ref{equivalence} this is equivalent to the 
statement of the recursion theorem.

For discussion about the motivation and applications of the 
ADN theorem we refer the reader to Visser~\cite{Visser} and 
Barendregt and Terwijn~\cite{BarendregtTerwijn}. 
For example, it has interesting applications in the theory 
of ceers (c.e.\ equivalence relations), see 
Bernardi and Sorbi~\cite{BernardiSorbi}. 
For recent results about diagonal functions for ceers see 
Badaev and Sorbi~\cite{BadaevSorbi}.

The following result simultaneously generalizes
Arslanov's completeness criterion (Theorem~\ref{Arslanov}) 
and the ADN theorem. 

\begin{theorem} {\rm (Joint generalization, Terwijn~\cite{Terwijn})} 
\label{joint}
Suppose $A$ is a c.e.\ set such that $A <_T \emptyset'$, 
and suppose that $\delta$ is a partial $A$-computable fixed point free 
function. Then for every partial computable function $\psi$ there 
exists a computable function $f$ totalizing $\psi$ avoiding $\delta$, 
i.e.\ such that for every~$n$, 
\begin{align}
\psi(n)\darrow \; &\Longrightarrow \; W_{f(n)}= W_{\psi(n)} \label{totalize2} \\
\psi(n)\uarrow \; &\Longrightarrow \; \delta(f(n))\uarrow   \label{avoid2}
\end{align}
\end{theorem}

Note that the statement of Theorem~\ref{joint} is identical 
to that of the ADN theorem, except that $\delta$ is now 
partial $A$-computable for $A$ c.e.\ and incomplete, instead 
of just p.c. 
So Theorem~\ref{joint} generalizes the ADN theorem in the 
same way that Arslanov's completeness criterion generalizes 
the recursion theorem.
Theorem~\ref{joint} implies Arslanov's completeness criterion, 
since in general $f$ as in the theorem cannot satisfy \eqref{avoid2} 
if $\delta$ is total. In particular, any total $A$-computable 
$\delta$ cannot be FPF, hence must have a fixed point. 
\begin{figure}[h]
\begin{center}
\setlength{\unitlength}{1.15mm}
\begin{picture}(24, 35)
\put(12,0){\makebox[0cm][c]{Recursion theorem}}
\put(7,4){\vector(-2,1){12}}
\put(17,4){\vector(2,1){12}}
\put(-6,15){\makebox[-.5cm][c]{\parbox[c]{2cm}{\centering ADN theorem\\ Theorem~\ref{ADN}}}}
\put(-6,20){\vector(2,1){12}}
\put(30,15){\makebox[-.5cm][c]{\parbox[c]{2cm}{\centering 
\; Arslanov \\ \; Theorem~\ref{Arslanov}}}}
\put(30,20){\vector(-2,1){12}}
\put(12,30){\makebox[0cm][c]{\makebox[-.5cm][c]{
\parbox[c]{3cm}{\centering Joint generalization Theorem~\ref{joint}}}}}
\end{picture}
\end{center}
\caption{Generalizations of the recursion theorem\label{fig2}}
\end{figure}
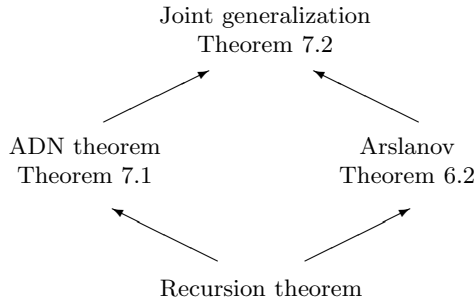

Visser actually proved the ADN theorem for arbitrary precomplete 
numberings, so that the ADN theorem also generalizes 
Ershov's recursion theorem.\footnote{
Not exactly the version with parameters Theorem~\ref{Ershov}, 
but the simpler statement \eqref{Ershov1} without parameters 
stated before it. 
That the ADN theorem with parameters fails was shown in 
Terwijn~\cite{Terwijn2020}.}
Arslanov's completeness criterion also holds for precomplete numberings, 
as was proved by Selivanov~\cite{Selivanov}.
So the obvious question at this point is whether 
this is also true for the joint generalization Theorem~\ref{joint}.
This is currently open (\cite[Question 5.2]{BarendregtTerwijn}).
Since it is true for the two theorems that the joint generalization 
generalizes, the evidence seems to point in the positive direction. 
However, the proof of the joint generalization uses specific 
properties of c.e.\ sets that we do not have in general, so 
that the answer may still be negative.

\section{Effectiveness and other remarks}

We  discuss the extent to which the various fixed point theorems 
discussed above are effective. 
As mentioned in section~\ref{sec:srt}, the recursion theorem 
is effective, which is the content of Theorem~\ref{recthm2}.
On the other hand, 
the proof of Arslanov's Theorem~\ref{Arslanov} does 
not effectively produce a fixed point, but rather an infinite 
c.e.\ set of numbers, at least one of which is a fixed point. 
That this is necessarily the case was discussed in 
Terwijn~\cite{Terwijn2020}, where it was shown that 
Theorem~\ref{Arslanov} indeed is not effective, so that there 
is no version with parameters analogous to Theorem~\ref{recthm2}.
This raises the question of exactly how noneffective 
Theorem~\ref{Arslanov} is. The matter of the complexity of 
the corresponding Skolem functions was discussed in 
Golov and Terwijn \cite{GolovTerwijn2022}, and 
independently in Arslanov~\cite{Arslanov2021}.


We already mentioned that the ADN theorem (Theorem~\ref{ADN}) is 
not effective, although there is uniformity in some of its
parameters. This was shown in \cite{Terwijn2020}. 
Since neither the Arslanov completeness criterion nor the 
ADN theorem are effective, a fortiori the same holds for 
the joint generalization Theorem~\ref{joint}.

It is not known whether the Arslanov completeness criterion 
(appropriately formulated) holds for pcas in general. 
See Terwijn~\cite[Question 10.1]{Terwijn2020b} for a precise 
statement of this.

The role that the recursion theorem plays in the theory of pcas is 
interesting. For example it plays an important part in 
results about {\em embeddings\/} between pcas, 
see Shafer and Terwijn \cite{ShaferTerwijn} and 
Golov and Terwijn \cite{GolovTerwijn2023}.
(Note that \cite{ShaferTerwijn} also contains results about another 
kind of fixed points, namely closure ordinals, but these are of a
different kind than the ones that we have been discussing here.)

We should also mention here the various forms of the recursion 
theorem in descriptive set theory, cf.\ Kechris~\cite[p289]{Kechris} 
and Moschovakis~\cite[p383]{Moschovakis}. 
These are formulated for the various pointclasses $\Gamma$ 
occurring in descriptive set theory (effective or not)  
for which the $\Gamma$-computable functions on Polish spaces 
can be suitably parameterized, e.g.\ by reals in $\omega^\omega$. 
For these an analog of the $S$-$m$-$n$-theorem is available 
(\cite[7A.1]{Moschovakis}), which makes the proof of the 
recursion theorem work.
Note that the idea of encoding continuous functions by reals 
is the same as the basic idea underlying Kleene's pca $\K_2$ 
(cf.\ \cite{vanOosten}).

Finally we mention categorical approaches to the subject of 
diagonalization and fixed point theorems, starting with 
Lawvere~\cite{Lawvere}.
Examples of Lawvere's basic scheme are discussed in 
Yanofsky~\cite{Yanofsky} 
and Bauer~\cite{Bauer}, among others.
Note, however, that these do not capture the more complex
results such as Theorem~\ref{Arslanov} and Theorem~\ref{joint}, 
as these do not follow Lawvere's scheme.

\subsection*{Acknowledgements}

We thank Dan Frumin and Anton Golov for discussions about 
the categorical view on the subject of fixed points.

\end{document}